\documentclass[11pt, oneside]{article}   	
\usepackage{geometry}                		
\usepackage{graphicx}				

\usepackage{amsmath,amsthm,amssymb}

\usepackage{fourier} 
\usepackage[scaled=0.875]{helvet} 

\usepackage[utf8]{inputenc}  
\usepackage[T1]{fontenc}
\usepackage{amsmath, amssymb, stmaryrd}
\usepackage[colorlinks, linkcolor=blue]{hyperref}
\usepackage[english]{babel}

\newcommand{\R}{\mathbb{R}}
\newcommand{\N}{\mathbb{N}}
\newcommand{\cW}{\mathcal{W}}

\newcommand{\bL}{\mathbb{L}}
\newcommand{\1}{\mathbf{1}}

\renewcommand{\P}{\mathbb{P}}

\newcommand{\tlog}{\mathfrak{log}}
\newcommand{\Var}{\operatorname{Var}}
\newcommand{\diam}{\operatorname{diam}}

\newtheoremstyle{mystyle}{}{}{\rmfamily}%
{}{\normalfont\bfseries}{ }{ }{} 

\newtheorem{theorem}{Theorem}[section]
\newtheorem{proposition}[theorem]{Proposition}
\newtheorem{lemma}[theorem]{Lemma}

\theoremstyle{mystyle}
\newtheorem{remark}[theorem]{Remark}

\title{\rule{\textwidth}{0.02cm}\\[.7cm]
\sf\bf On the law of the iterated logarithm\\ for continued fractions with \\ sequentially restricted partial quotients \\ \bigskip}
\author{{\small \bf Manuel Stadlbauer} \\ {\footnotesize Instituto de Matem\'atica }\\{\footnotesize  Universidade Federal do Rio de Janeiro}\\{\footnotesize  21941-909 Rio de Janeiro (RJ), Brazil}
 \and {\small \bf Xuan Zhang}\\ {\footnotesize Instituto de Matem\'atica e Estat\'istica}\\ {\footnotesize Universidade de S\~ao Paulo}\\ {\footnotesize 05508-090 S\~ao Paulo (SP), Brazil.} }

\date{\small January 23, 2020\\[-0.1cm] \rule{\textwidth}{0.02cm}}

\begin{document}
\maketitle

\begin{abstract} 
We establish a law of the iterated logarithm (LIL) for the set of real numbers whose $n$-th partial quotient is bigger than $\alpha_n$, where $(\alpha_n)$ is a sequence such that $\sum 1/\alpha_n$  is finite. This set is shown to have Hausdorff dimension $1/2$ in many cases and the measure in LIL is absolutely continuous to the Hausdorff measure. The result is obtained as an application of a strong invariance principle for unbounded observables on the limit set of a sequential iterated function system.    
\end{abstract}

\section{Introduction and statement of the main result}
The law of the iterated logarithm  for continued fractions provides an upper bound for the almost sure growth of the partial quotients of $x \in [0,1] \setminus \mathbb Q$. That is, with the  partial quotients  $x_0,x_1, \ldots \in \mathbb N$  given by
\[ x = \frac{1}{x_0 +  \frac{1}{\textstyle x_1 + \frac{1}{\textstyle x_2 + \cdots}}}, \]
the law of the iterated logarithm for $(x_n)$ states that, for Lebesgue almost every $x$, 
\begin{equation} \label{eq:Classical LIL for continued fractions} \limsup_{n\to \infty} \frac{|\sum_{k=0}^{n-1} \log (x_k/a) |}{\sqrt{n\log\log n }} = b,\end{equation}
for some constants $a,b>0$, where $a$ is Khinchin's constant, the geometric mean of $x_k$, and $b$ can be seen as the variance of $\log x_k$. The result is due to Stackelberg (\cite{Stackelberg:1966a}) and was extended shortly thereafter by substituting  $\log (x_k/a)$ with more general functions and $b$ with the associated variance: Reznik for a class of unbounded, continuous  functions in $L^{2+\epsilon}$ (\cite{Reznik:1968a}), Philipp  for functions in $\bigcap_{p>1} L^p$ (\cite{Philipp:1969a}) and Gordin \& Reznik  for the log of the denominator of the $n$-th approximant (\cite{GordinReznik:1970a}). Moreover, a predecessor of \eqref{eq:Classical LIL for continued fractions} was obtained by Doeblin (\cite{Doeblin:1940a}).

The aim of this note is to study subsets $X_\alpha$ of $[0,1]$ with prescribed minimal growth of partial quotients and establish analogous results. 
That is, for $\alpha \in [0,1]\setminus \mathbb{Q}$ with partial quotients  $(\alpha_n)$, set  $X_\alpha := \{x \in [0,1] : x_n \geqslant \alpha_n \ \forall n \in \mathbb N_0\}$.
These sets are of fractal nature and their Hausdorff dimension was studied by several authors. Good (\cite{Good:1941}) showed that the Hausdorff dimension $\dim_{\hbox{\tiny H}} (X)=1/2$, where $X=\{x: x_n\to\infty\}$,
whereas \L{}uczak (\cite{L-uczak:1997}) showed that, if $\alpha_n = c^{b^n}$ with $b, c>1$, then  $\dim_{\hbox{\tiny H}} (X_\alpha) = 1/(b+1)$ (for partial results, see \cite{Hirst:1970,Moorthy:1992a}). Note that these results are similar in spirit but not directly related to results for sets defined by  $x_n \geqslant \alpha_n$ infinitely often as in the Borel-Bernstein theorem or in \cite{WangWu:2008a}. In there, Wang \& Wu obtained an almost complete description of the Hausdorff dimensions for these sets. 

In here, we aim for an almost sure description of the growth of the partial quotients of the elements in $X_\alpha$. In order to do so, we construct a relevant measure, establish the law of the iterated logarithm and finally specify the class of this measure.  We obtain the following analogue of Stackelberg's result. In here, we will write $ a_n \ll b_n $ if there exists $C$ such that $ a_n \leqslant C b_n $.

\begin{theorem}\label{theo:easy_bounds}Assume that $\sum_n 1/\alpha_n < \infty$  
and let $\gamma_n$ be given by $\gamma_n \alpha_n^{\gamma_n}=1$. Then 
there exists a probability measure $\nu$ on $X_\alpha$ for which the following holds almost surely.
 \begin{align*}
\limsup_{n\to \infty} \frac{\sum_{k=0}^{n-1} (\gamma_k \log (x_k/\alpha_k) - 1)}{\sqrt{2 n\log\log n }} = 1.
\end{align*}
Furthermore, if $  n\ll \alpha_n \ll  \lambda^{n}$ or $\lambda^{n} \ll \alpha_n \ll  n^{n(1-\epsilon)}$ for some $\lambda >1$ and $\epsilon >0$, then $\nu$ is absolutely continuous with respect to the $1/2$-dimensional Hausdorff measure $\mu_{1/2}$, $\dim_{\hbox{\tiny H}}(X_\alpha) = 1/2$,  and $\mu_{1/2}(X_\alpha) = \infty$. 
\end{theorem} 

Note that $\gamma_n$ is of the order of $\log \log \alpha_n /\log \alpha_n$ (Remark \ref{remark:alpha vs gamma}). We also would like to emphasize that the rescaling of the Birkhoff sum with $(\gamma_n)$ is a consequence of the fractal nature of  $X_\alpha$. Roughly speaking, the $\gamma_n$ are chosen according to  consecutive applications of Moran's formula for the Hausdorff dimension (or Bowen's formula in the language of thermodynamic formalism). That is, the $\gamma_n$ are chosen such that the weights attributed to the intervals of $\{x: x_k \geq \alpha_k, 1 \leq k \leq n \}$ remain summable (see Lemma \ref{Lemma:Estimate_L(1)}). Moreover, the geometric nature of this construction is underlined by the fact that the measure $\nu$  associated to this choice is absolutely continuous with respect to the Hausdorff measure.
%

In fact, we show a stronger statement as in Theorem \ref{theo:easy_bounds}. That is, we prove an abstract strong invariance principle for sequences of locally H\"older continuous functions with finite fourth moment (Theorem \ref{theo:invpp}) and obtain as corollaries a slightly more general law of the iterated logarithm and a law of large numbers, that is  $\sum_{k=0}^{n-1}  \gamma_k \log (x_k/\alpha_k) / n$ converges almost surely to $1$  
(Theorem \ref{prop:Khinchine-Stackelberg}). The summability of $(1/\alpha_n)$  is required  
to control the fluctuations of the variance (Lemma \ref{lemma:Estimate_log}). Moreover, this bound even allows to determine the asymptotic variance, a quantity which in most situations is not accessible explicitly.    
The other bounds are of geometric origin and are required to identify $\nu$ as a geometric measure by showing that $d\nu/d\mu_{1/2}=0$. It is also worth noting that the upper bound is optimal in the sense that $\dim_{\hbox{\tiny H}}(X_\alpha) = 1/(b+1) < 1/2$ for $\alpha_n = c^{b^n}$ and $b, c >1$, as shown by \L{}uczak in \cite{L-uczak:1997}.

From the viewpoint of dimension theory, the novelty in here is the absolute continuity as the result on the Hausdorff dimension is an immediate corollary of \cite{Good:1941,L-uczak:1997} or might be deduced from \cite{Rempe-GillenUrbanski:2016}. It is worth noting that the proof of  absolute continuity requires the above law of large numbers, and that absolute continuity allows to formulate a result independent of $\nu$. Namely, as $d\nu/d\mu_{1/2}=0$,
the laws of the iterated logarithm and large numbers hold on a dense set of infinite $\mu_{1/2}$-measure.
Furthermore, 
due to the flexibility arising from local H\"older continuity, the result also is applicable to the denominators of the approximants $q_n(x)$ given by $p_n(x)/q_n(x) = 1/(x_0 + 1/(x_1 + 1/( \cdots + 1/x_{n-1})\cdots))$ in the second part of  Theorem \ref{prop:Khinchine-Stackelberg}. 

The method of proof might be summarised as follows. As a first step, we consider the sequence $(X_n)$ of shift spaces $X_n := \{(x_j): x_j \geqslant \alpha_{n+j}\ \forall j\geqslant 0\}$ equipped with a sequence of Hölder continuous functions $\varphi_n: X_n  \to (0,\infty)$. As in \cite{BessaStadlbauer:2016}, we define a family of ratio operators $\{\P_n^k\}$ such that $\P_n^k$ maps functions defined on $X_n$ to functions on $X_{n+k}$ with $\P_n^k(\mathbf 1)=\mathbf 1$  and show that its dual uniformly contracts the Wasserstein distance (Theorem \ref{theo:exponential decay}). As a corollary one obtains a unique sequence of probability measures $(\nu_n)$ with $(\P_n^k)^\ast(\nu_{n+k})= \nu_n$ and a spectral gap for $\P_n^k$ acting on $\nu_n$-integrable functions with $\nu_n$-integrable local Hölder coefficients (Theorem \ref{theo:exponential decay for L1}). 
With these preparations, the strong invariance principle of Cuny \& Merlev\`ede (\cite{CunyMerlevede:2015}) is applicable and gives rise to a law of the iterated logarithm for functions with finite fourth moment (Theorem \ref{theo:invpp}). It then remains to verify that this result is applicable to continued fractions. As a first step, it turns out that it is necessary to consider the potential $\varphi_n(x) := \varphi(x)^{(\gamma_n+1)/2}$, where $\varphi$ refers to the geometric potential. Therefore, the canonical approach in this context is to consider sequences of probability measures $(\nu_n)$ instead of a single reference measure. Moreover, as $(\alpha_n)$ does not grow too slow, the asymptotic behavior of the integrals with respect to $(\nu_n)$ is accessible (Lemma \ref{lemma:Estimate_log}), which then allows to determine the asymptotic variance of the relevant, unbounded observables. The law of large numbers then allows to prove absolute continuity with respect to the Hausdorff measure, provided that $(\alpha_n)$ does not grow too fast (Proposition \ref{prop:HD-dimension}). However, as $\gamma_n$ tends to $0$, it is unclear if it is possible to extend the approach to $\alpha$ with $\dim_{\hbox{\tiny H}} (X_\alpha) < 1/2$.     

The main novelty of proof is based on the fact that the problem under consideration does not admit a common reference measure by construction. Therefore, there is no common invariant function space to the associated transfer operators, implying that important tools from spectral theory like quasi-compactness are not immediately applicable. However, based on an idea by Hairer \& Mattingly (\cite{HairerMattingly:2008}), it is possible to replace the concept of a spectral gap by a contraction of the Vaserstein metric on probability measures. The idea was adapted to non-stationary shift spaces in \cite{Stadlbauer:2015} for normalised potentials and then refined in \cite{BessaStadlbauer:2016} by considering quotients of operators. In here, we combine these and extend them in Theorem \ref{theo:exponential decay for L1} to functions in $L^1(\nu_n)$ with integrable local H\"older coefficients. This is of importance as we are now able to include unbounded functions (like $\log x_k$) in our analysis and do not have to take care about invariant measures which do not exist in the context of non-stationary shift spaces (see \cite{Fisher:2009} or \cite[Prop. 2.2]{BessaStadlbauer:2016}). 
The method based on the contraction of the Wasserstein metric is significantly different to the recent approaches to fiberwise quasi-compactness for random operators by Dragicevic, Froyland, González-Tokman \& Vaienti in \cite{DragicevicFroylandGonzalez-Tokman:2016a,DragicevicFroylandGonzalez-Tokman:2017a} as 
 non-stationary operators by definition do not have a spectrum.
 
For families of maps acting on the same space with respect to a common reference measure, there are several known results for central limit theorems and invariance principles. Namely, Heinrich (\cite{Heinrich:1996}) and Conze \& Raugi (\cite{ConzeRaugi:2007}) obtained central limit theorems for sequential dynamical systems given by a family of expanding maps of the interval. This result was improved recently through application of the strong invariance principle of  Cuny \& Merlev\`ede as above by Haydn, Nicol, Török \& Vaienti in the context of sequential expanding maps (\cite{HaydnNicolTorok:2017}) and by Nicol, Török \& Vaienti (\cite{NicolTorokVaienti:2016}) for sequential intermittent systems.

\section{Non-stationary shift spaces and decay of correlations}

In this section, we recall the definition of the non-stationary full shift, the associated family of Ruelle operators and related results. In here, we closely follow the notation in \cite{Stadlbauer:2015,BessaStadlbauer:2016}. For each $n \in \N = \N \cup \{0\}$, choose a subset $\cW_n \neq \emptyset$ of $\N$. We  refer to the sequence shift spaces $(X_n)$ and maps $(T_n)$ given by  
\begin{align*}
X_n &:=\left\{ (x_n, x_{n+1},\ldots):\, x_i \in \cW_i \,\,\forall i=n,n+1,\ldots
\right\},\\ 
T_n &:  X_n \to X_{n+1}, \quad (x_n, x_{n+1},\ldots ) \mapsto  (x_{n+1},x_{n+2},\ldots ) 
\end{align*}
as the \emph{non-stationary full shift} associated with $(\cW_n)$. In complete analogy to classical shift spaces, we will refer to $\cW_n^k := \left\{ (x_n, x_{n+1},\ldots, x_{n+k-1}): x_i \in \cW_i \right\}$ as the set of words of length $k$ at position $n$ and, for $w = (x_n, x_{n+1},\ldots, x_{n+k-1})\in \cW_n^k$, and to
\[ [w] := \left\{ (y_n, y_{n+1},\ldots) \in X_n:\, y_i=x_i, \,\forall i=n,n+1,\ldots, n+k-1
\right\} \subset X_n\]
as the cylinder associated to $w$. Let $\mathcal F_n$ be the $\sigma$-algebra generated by all cylinders in $X_n$. The iterates are maps from $X_n$ to $X_{n+k}$, which are given by $T_n^k := T_{n+k-1} \circ\cdots\circ T_n$. Furthermore, as  $T_n^k |_{[w]}$ is invertible for each $w \in \cW_n^k$, its inverse $\tau_w:X_{n+k} \to [w]$ is well defined. 
As in case of stationary shifts, the $X_n$ are Polish spaces with respect to the metric 
\[
d_r((x_0,x_1,x_2,\ldots),(y_0,y_1,y_2,\ldots)) := r^{\min\{i: x_i \neq y_i\} }
\] 
for $r \in (0,1)$. The value of $r$ is fixed throughout the paper. This metric is compatible with the topology 
generated by cylinder sets and the $T_n$ are uniformly expanding as $d_r(T_n(x),T_n(y)) = r^{-1} d_r(x,y)$.
We now introduce the relevant function spaces and operators. For  $f:X_n \to \R$ and $r \in \R$, we refer to  
 \[ D_r(f) := \sup \left\{ |f(x)-f(y)|/r^{k} : x,y \in [w], w \in \mathcal{W}_n^k, k=1,2, \ldots    \right\} \]
 as the $r$-H\"older coefficient of $f$. The space of $r$-H\"older functions is defined by 
 \[ \mathcal{H}^r_n := \{f : \|f \|_{\mathcal{H}} < \infty\}, \hbox{ with } \|f \|_{\mathcal{H}} := \|f\|_\infty + D_r(f).\]
In comparison, we denote the Lipschitz norm by $$D(f):=\sup\left\{ |f(x)-f(y)|/d_r(x,y) : x,y \in X_n   \right\}.$$
Note that since $X_n$ has finite diameter, $D(f)\leqslant 2\|f\|_\mathcal H$ and on the other hand if $D(f)$ is finite then $\|f\|_\mathcal H$ is also finite. We say that a sequence $\{f_n\}$ has  \emph{uniformly bounded H\"older coefficients} if there exists $r \in (0,1)$ such that $\sup_{n \in \N_0} D_r(f_n) < \infty$. Now assume that, for each $n\in \N_0$, there is a given (multiplicative) potential $\varphi_n:X_n \to (0,\infty)$ such that $\{\log \varphi_n\}$ has uniformly bounded H\"older coefficients. Then the $n$-th Ruelle operator is defined by, 
 for $x \in X_{n+1}$ and  $f:X_{n} \to [0,\infty)$ 
\[ \bL_{n}(f)(x) = \sum_{w \in \cW_n} \varphi_n(\tau_w(x)) f(\tau_w(x)).\]
By adapting standard arguments, it easily can be seen that, if $\|\bL_{n}(\1)\|_\infty < \infty$, then $ \bL_{n}$ acts on $r$-H\"older functions. 
In analogy to $T_n^k$, set $\bL_n^k := \bL_{n+k-1} \circ \cdots \circ \bL_n$. Furthermore, set
\begin{equation}\nonumber
\label{eq:ratio-operators}
\P^k_{n}(f):= 
\frac{\bL_{n}^k( f \cdot \bL_{0}^n(\1))}{\bL_0^{n+k}( \1)}.
\end{equation}
By following the lines of the proof of Lemma 2.1 in \cite{BessaStadlbauer:2016}, it  follows that 
$\P^k_{n}: \mathcal{H}^r_n \to \mathcal{H}^r_{n+k}$, provided that $\|\bL_{m}(\1)\|_\infty < \infty$ for $n \leqslant m < n+k$. Furthermore, $\P^k_{n}(\1)=\1$ and a Doeblin-Fortet inequality holds. In fact, the following statement about decay of correlations holds, which is a straightforward adaption of Theorem 2.1 in \cite{BessaStadlbauer:2016} to the setting of countable shift spaces. 

\begin{theorem}\label{theo:exponential decay}
 Assume that $(\log \varphi_n: n \in \N_0)$ has uniformly bounded H\"older coefficients and that  
$\|\bL_{n}(\1)\|_\infty < \infty$ for all $n \in \N_0$. Then there exist $C \geqslant 1$, $s \in (0,1)$ and Borel probability measures $\nu_n$ on $(X_n,\mathcal F_n)$ for $n \in \N_0$ such that  
for all $f \in  \mathcal{H}^r_n$ 
\begin{enumerate}
\item $D(\P_{n}^k(f)) \leqslant C s^k D(f), \forall k\in\mathbb N_0$.
\item $\| \P_{n}^k(f) - \int f d\nu_n\1\|_\infty \leqslant C s^k D(f), \forall k\in\mathbb N_0$.
\item $\int f d\nu_n = \lim_{k\to \infty} \P_{n}^k(f) (x_k)$ for any sequence $(x_k)$ with $x_k \in X_{n+k}$.
\end{enumerate}
\end{theorem}

\begin{proof} Observe that the proof of Theorem 2.1 in \cite{BessaStadlbauer:2016} 
in verbatim applies to our setting. In particular, the first assertion is a consequence of (2.5) in  
\cite{BessaStadlbauer:2016}. Furthermore, the action on probability measures by the dual of $\P^k_{n}$ for $k$ sufficiently large is a strict contraction with respect to the Wasserstein metric ((2.4) in \cite{BessaStadlbauer:2016}). The existence of $\nu_n$ as defined in (3) follows from that. Hence, it remains to show (2). In order to do so, observe that $(\P_n^k)^\ast(\nu_{n+k})= \nu_{n}$ as $ \P_{n+k}^l \circ \P_n^k =  \P_{n}^{l+k}$. The second assertion then follows from Kantorovich's duality applied to the measure $\nu_{n+k}$ and the Dirac measure supported on some point in $X_{n+k}$. 
  \end{proof}

In order to include the functions of interest in the context of continued fractions, which are unbounded and therefore not in $\mathcal{H}^r_n$, we extend Theorem \ref{theo:exponential decay} to the following function space. For  $f:X_n \to \R$ and $r \in \R$, we consider the locally constant function $\mathcal{D}_r(f): X_n \to [0,\infty]$ defined by, for every $x \in  [a]$, $a \in \mathcal{W}_n^1$,  
 \[\mathcal{D}_r(f)(x) :=  \sup \left\{ |f(y_1)-f(y_2)|/r^{k} : y_1,y_2 \in [a]\cap  [w], w \in \mathcal{W}_n^k, k=1,2, \ldots    \right\} .\]
Note that, by construction, $\mathcal{D}_r(f)$ is constant on cylinders of length 1.  
 Furthermore, set $\|f\|_{n,p}:= (\int |f|^p d\nu_n)^{1/p}$. Clearly $\|\mathcal D_r(f)\|_{n,1}\leqslant D_r(f)$.  Define the space of functions with integrable local $r$-H\"older coefficients by
 \[ \mathcal{H}^{\tiny\hbox{\it loc}}_n := \left\{f: X_n \to \R \,\big| \, \|f\|_{n,1} < \infty, \|\mathcal{D}_r(f)\|_{n,1} < \infty \right\}.\]
\begin{theorem}\label{theo:exponential decay for L1} Under the assumptions of Theorem \ref{theo:exponential decay}, there exist $C>1$ and $s\in (0,1)$ such that for all  $f \in  \mathcal{H}^{\tiny\hbox{\it loc}}_n$ and  $k \in \N$, 
\begin{enumerate}
\item $D(\P_{n}^k(f)) \leqslant C s^k (\|f\|_{n,1}  + \|\mathcal{D}_r(f)\|_{n,1})$,
\item $\| \P_{n}^k(f) - \int f d\nu_n\1\|_\infty \leqslant  C s^k (\|f\|_{n,1}  + \|\mathcal{D}_r(f)\|_{n,1})$.
\end{enumerate}
\end{theorem}
The theorem immediately follows from applying Theorem \ref{theo:exponential decay} to $\P_n^1(f)$ and the following Lemma showing that  $\P_{n}^1$ maps   $ \mathcal{H}^{\tiny\hbox{\it loc}}_n$ into $\mathcal{H}^r_{n+1}$.
\begin{lemma} There exists $C>1$ such that for all $f \in  \mathcal{H}^{\tiny\hbox{\it loc}}_n$, 
\begin{enumerate}
\item $\| \P_{n}^1(f)\|_\infty \leqslant  C \left( \|f\|_{n,1}   + \|\mathcal{D}_r(f)\|_{n,1}\right)$,
\item $D_r(\P_{n}^1(f)) \leqslant C \left(\|f\|_{n,1}    + \|\mathcal{D}_r(f)\|_{n,1}\right)$.
\end{enumerate}
\end{lemma}

\begin{proof}
Let $g$ be a non-trivial, non-negative function in $\mathcal{H}^{\tiny\hbox{\it loc}}_n$ which is constant on cylinders, that is, $\mathcal{D}_r(g)=0$. It follows from bounded distortion and $T_n([a])=X_{n+1}$ for all $a \in \cW_n^1$ that $\bL_{n}(g)(x) \asymp \bL_{n}(g)(y)$ for all $x,y \in X_n$. Hence,  $\P_{n}^1(g)(x) \asymp \P_{n}^1(g)(y)$ and 
\[ \P_{n}^1(g)(x) \asymp \int \P_{n}^1(g)(x) d\nu_{n+1}(x) =  \int  g d\nu_{n}.\] 
The first statement  follows from bounding $f$ from above and below by functions $f_+,f_-$ which are constant on cylinders such that $ f- \mathcal{D}_r(f) \leqslant f_- \leqslant f \leqslant f_+ \leqslant  f + \mathcal{D}_r(f)$ and using that  $\mathcal{D}_r(f)$ is by construction constant on cylinders. 

For the remaining statement, observe that by uniform H\"older continuity of $\log \varphi_n$, we have that    
\begin{align*}
& \;  \left| \bL_{n}(f  \cdot \bL_{0}^n(\1))(x) - \bL_{n}(f \cdot \bL_{0}^n(\1))(y) \right| / d(x,y)\\
\leqslant &\; \frac{1}{d(x,y)} \left(\sum_{w \in \cW_n} \varphi_n(\tau_w(x))\left|1 - \frac{\varphi_n(\tau_w(y))}{\varphi_n(\tau_w(x))}\right| \cdot |f \cdot \bL_{0}^n(\1)|(\tau_w(x)) \right.\\
 & +  \sum_{w \in \cW_n}  \varphi_n(\tau_w(y)) \bL_{0}^n(\1)(\tau_w(x)) |f(\tau_w(x)) -  f(\tau_w(y))|  \\
 & +  \left. \sum_{w \in \cW_n}  \varphi_n(\tau_w(y)) | f(\tau_w(y))| \bL_{0}^n(\1)(\tau_w(y))
 	\left|\frac{\bL_{0}^n(\1)(\tau_w(x))}{\bL_{0}^n(\1)(\tau_w(y))} -1 \right| \right)\\
\leqslant & \;  C\left(  \bL_{n}(|f| \cdot \bL_{0}^n(\1))(x) + r \bL_{n}( \mathcal{D}_r(f) \cdot  \bL_{0}^n(\1))(x) +  r\bL_{n}(|f| \cdot \bL_{0}^n(\1))(y) \right) .
\end{align*}
It easily follows from this that $  {D}_r(\P_n^1(f)) \ll   \P_{n}^1(|f|)(x) +  \P_{n}^1(\mathcal{D}_r(f))(x) \ll \|f\|_{n,1}    + \|\mathcal{D}_r(f)\|_{n,1} $.
  \end{proof}

We collect some  properties of the operator $\P_n^k$ and the measure $\nu_n$. We will use the results in this section freely without further references, and always denote by $\mathbb E$ and $\Var$ the expectation and the variance with respect to $\nu_0$.
\begin{proposition}The following equations hold for all $f\in  \mathcal{H}^{\tiny\hbox{\it loc}}_{n+k}, g\in \mathcal{H}^{\tiny\hbox{\it loc}}_n$.
\begin{enumerate}
\item $\P_n^{k}(f\circ T^k_n\cdot g)=f\cdot \P_n^{k}(g)$.
\item $\int f\circ T^k_n\cdot g~d\nu_n=\int f\cdot \P_n^{k}(g)~d\nu_{n+k}$.
\item $\mathbb E(g\circ T^n_0|(T^{n+k}_0)^{-1}\mathcal F_{n+k})=(\P_n^{k}(g))\circ T_0^{n+k}.$
\end{enumerate}
\end{proposition}
\begin{proof}
\begin{enumerate}
\item $\P_n^k(f\circ T_n^k\cdot g)=\frac{\bL_{n}^k( f\circ T_n^k\cdot g \cdot \bL_{0}^n(\1))}{\bL_0^{n+k}( \1)}=\frac{f\cdot \bL_n^k(g\cdot \bL_0^n(\1))}{\bL_0^{n+k}(\1)}=f\cdot \P_n^k(g).$
\item This is due to $(\P_n^k)^*\nu_{n+k}=\nu_n$ and (1).
\item Let $A=T_0^{-(n+k)}A_{n+k}$ for some $A_{n+k}\in \mathcal F_{n+k}$, then by (2)
\begin{align*}
\int_A g\circ T^n_0 ~d\nu_0&=\int g\cdot {\bf 1}_{A_{n+k}}\circ T_n^{k}~d\nu_{n}=\int \P_n^k(g)\cdot {\bf 1}_{A_{n+k}}~d\nu_{n+k}\\
&=\int_A (\P_n^k(g))\circ T_0^n~d\nu_0. 
\end{align*} 
\end{enumerate}
\end{proof}
\section{An almost sure invariance principle}
We shall prove an almost sure invariance principle for the non-stationary full shift $(X_n,\mathcal F_n, T_n, \nu_n)$ described above. The proof is based on the almost sure invariance principle for reverse martingale differences by Cuny and Merlev\`ede.

\begin{theorem}[{\cite[Theorem 2.3]{CunyMerlevede:2015}}]\label{thm:rvsmtg}
Let $(U_n)_{n\in\mathbb N}$ be a sequence of square integrable reverse martingale differences with respect to a non-increasing filtration $(\mathcal G_n)_{n\in\mathbb N}$. Assume that $\sigma_n^2:=\sum_{k=1}^n\mathbb E(U_k^2)\to\infty$ and that $\sup_n\mathbb E(U_n^2)<\infty$. Let $(a_n)_{n\in\mathbb N}$ be a non-decreasing sequence of positive numbers such that $(a_n/\sigma_n^2)_{n\in\mathbb N}$ is non-increasing and $(a_n/\sigma_n)_{n\in\mathbb N}$ is non-decreasing. Assume that
\begin{align*}
&\sum_{k=1}^n\left(\mathbb E(U_k^2|\mathcal G_{k+1})-\mathbb E(U^2_k)\right)=o(a_n) \qquad \mathbb P\text{-a.s.}\\
&\sum_{n\geqslant 1}a_n^{-\nu}\mathbb E\left(|U_n|^{2\nu}\right)<\infty \qquad \text{for some } 1\leqslant\nu\leqslant 2.\label{eq:sumvar}
\end{align*}
Then, enlarging our probability space if necessary, it is possible to find a sequence $(Z_k)_{k\geqslant 1}$ of independent centered Gaussian variables with $\mathbb E(Z_k^2)=\mathbb E(U_k^2)$ such that
$$\sup_{1\leqslant k\leqslant n}\left|\sum_{i=1}^k U_i-\sum_{i=1}^k Z_i\right|=o((a_n(|\log (\sigma_n^2/a_n)|+\log\log a_n))^{1/2})\qquad \mathbb P\text{-a.s.}$$
\end{theorem}

As an application, we obtain the following invariance principle for non-stationary shift spaces whose rate of approximation is directed towards a law of the iterated logarithm.

\begin{theorem}\label{theo:invpp} Let $\{f_n\in \mathcal{H}^{loc}_n\}_{n\in\mathbb N_0}$ be a sequence of functions with $\int f_n d\nu_n =0$ for all $n \in \N_0$,  
\[ \sup_n\int f_n^4d\nu_n<\infty \quad \text{and} \quad \sup_n \int (\mathcal{D}_r(f_n))^2 d\nu_n < \infty. \]
Furthermore, for $s_n^2 := \Var(\sum_{k=0}^{n-1} f_k\circ T_0^k)$, assume that 
$s_n\to \infty$ and $\sum_n s_n^{-4}<\infty$. Then, enlarging our probability space if necessary, there exists a sequence $(Z_n)$ of independent centered Gaussian random variables such that 
\begin{gather*}
\sup_n\left|\sqrt{\textstyle \sum_{k=0}^{n-1} \Var(Z_k)}-s_n\right|<\infty,\\
\sup_{0\leqslant k \leqslant n-1} \left| \textstyle\sum_{i=0}^k f_i\circ T_0^i - \sum_{i=0}^k Z_i \right| = o(\sqrt{s^2_n \log \log s^2_n}) \hbox{ a.s.}.
\end{gather*}
\end{theorem}

Before giving the proof of the theorem, we would like to remark that, in order to obtain a central limit theorem from the almost sure invariance principle, one would have to improve the order of the approximation (see, for example, \cite{HaydnNicolTorok:2017} for BV-observables on the unit interval).
This would require more assumptions on $(f_n)$ which are probably satisfied in the setting of sequentially restricted partial quotients. However, in this paper, we only deal with almost sure results, partly because this is the natural approach for considering continued fractions with restricted quotients with respect to  
the Hausdorff measure, which is the reference measure in our setting. Namely,  $\nu_0$ is absolutely continuous with respect to this measure (see Proposition \ref{prop:HD-dimension}) and therefore, an almost sure result with respect to $\nu_0$ holds for a set of positive Hausdorff measure.

\begin{proof}
Denote $\mathcal G_n:=(T_0^n)^{-1}\mathcal F_n$ for $n\in\mathbb N$ and $\mathcal G_0:=\mathcal F_0$, then $\mathcal G_n$ is a non-increasing filtration. Let $h_0:=0$, and define $h_n\in \mathcal{H}^r_n$ recursively by $h_{n+1}=\P_n^1f_n+\P_n^1h_n$. Then $h_n=\sum_{k=0}^{n-1}\P_k^{n-k}f_k$. It follows from Theorem \ref{theo:exponential decay for L1} that for some constants $C\geqslant 1$ and $s\in(0,1)$ 
$$\|h_n\|_\infty\leqslant \sum_{k=0}^{n-1}\|\P^{n-k}_k f_k\|_\infty\leqslant C\sum_{k=0}^{n-1}s^{n-k}(\|f_k\|_{k,1}+\|\mathcal{D}_r(f_k)\|_{k,1})
$$
and
$$D(h_n)\leqslant \sum_{k=0}^{n-1}D(\P^{n-k}_k f_k)\leqslant C \sum_{k=0}^{n-1}s^{n-k}(\|f_k\|_{k,1}+\|\mathcal{D}_r(f_k)\|_{k,1}).
$$
They are both uniformly bounded by the assumption.
Let $$u_n:=f_n+h_n-h_{n+1}\circ T_n^{1}\in \mathcal H^r_n, \quad U_n:=u_n\circ T_0^n.$$ Clearly, $U_n$ is $\mathcal G_n$-measurable and square integrable. Furthermore, because 
\begin{align*}\mathbb E(U_n|\mathcal G_{n+1})&=\mathbb E(f_n\circ T_0^n|\mathcal G_{n+1})+\mathbb E(h_n\circ T_0^n|\mathcal G_{n+1})-h_{n+1}\circ T_0^{n+1}\\&=(\P_n^1 f_n)\circ T_0^{n+1}+(\P_n^1 h_n)\circ T_0^{n+1}-h_{n+1}\circ T_0^{n+1}=0,\end{align*} 
$(U_n)_{n\in\mathbb N_0}$ is a sequence of square integrable reverse martingale differences. Let $$\sigma_n^2:=\sum_{k=0}^{n-1}\int u_k^2d\nu_k=\sum_{k=0}^{n-1}\mathbb E(U_k^2)=\mathbb E\left(\left(\sum_{k=0}^{n-1}U_k\right)^2\right).$$ We check the conditions of Theorem \ref{thm:rvsmtg} with $a_n=\sigma_n^2$. 
First we show $\sigma_n^2\to\infty$ and $\sup_n\mathbb E(U_n^2)<\infty$. It follows from 
$$|\sigma_n-s_n|=\left|\left\|\sum_{k=0}^{n-1}U_k\right\|_2-\left\|\sum_{k=0}^{n-1}f_k\circ T_0^k\right\|_2\right|\leqslant \left\|\sum_{k=0}^{n-1}U_k-\sum_{k=0}^{n-1}f_k\circ T_0^k\right\|_2=\|h_n\circ T_0^n\|_2 $$
that $|\sigma_n-s_n|$ is uniformly bounded. Meanwhile $s_n^2\to\infty$ implies that $\sigma_n^2\to\infty$. 
The uniform bound for $\mathbb E(U_n^2)$ follows from the uniform bounds for $\|f_n\|_{n,2}$ and $\|h_n\|_\infty$.

Next we show that $$\sum_{k=0}^{n-1}\left(\mathbb E(U_k^2|\mathcal G_{k+1})-\mathbb E(U^2_k)\right)=o(\sigma^2_n) \qquad \mathbb \nu_0\text{-a.s.}$$
We use ideas from the proof of \cite[Theorem 4.1]{ConzeRaugi:2007}.
Let $\tilde u_k:=u_k^2-\int u_k^2 d\nu_{k}$ and $$F_n:=\sigma_n^{-2}\sum_{k=0}^{n-1}\mathbb E(U_k^2|\mathcal G_{k+1}).$$ 
Then $$\sum_{k=0}^{n-1}\mathbb E(U_k^2|\mathcal G_{k+1})-\mathbb E(U_k^2)=\sum_{k=0}^{n-1}\P_k^1 \tilde u_{k}\circ T_0^{k+1}=\sigma_n^2(F_n-1).$$ It remains to show that $F_n-1\to 0$ a.s.
Note that for some constants $C>1$ and $s\in (0,1)$, we have 
\begin{align*}
\mathbb E\left(\sum_{k=0}^{n-1} \P_k^1\tilde u_k\circ T_0^{k+1}\right)^2
&\leqslant 2 \sum_{0\leqslant k\leqslant l\leqslant n-1}\mathbb E (\P_k^1\tilde u_k\circ T_0^{k+1}\cdot \P_l^1\tilde u_l\circ T_0^{l+1})\\
&=2\sum_{0\leqslant k\leqslant l\leqslant n-1}\int(\P_k^{l-k+1}\tilde u_k\cdot \P_l^1\tilde u_l)d\nu_{l+1}\\
&\leqslant C \sum_{0\leqslant k\leqslant l\leqslant n-1}s^{l-k+1} (\| \tilde u_k\|_{k,1}+\|\mathcal D_r(\tilde u_k)\|_{k,1})\cdot \|\P_l^1\tilde u_l\|_{l+1,1} \\
&\leqslant 2C\sum_{0\leqslant k\leqslant l\leqslant n-1}s^{l-k+1}(\mathbb E(U_k^2)+\|\mathcal D_r(u_k^2)\|_{k,1})\cdot \mathbb E (U_l^2).
\end{align*}
In this upper bound $\mathbb E(U_k^2)$ is uniformly bounded, and we will show the same for $\|\mathcal D_r(u_k^2)\|_{k,1}$. Write out
$$u_k^2=f_k^2+h_k^2+h^2_{k+1}\circ T_k^1+2f_kh_k-2f_k\cdot(h_{k+1}\circ T_k^1)-2h_k\cdot(h_{k+1}\circ T_k^1).$$
For any $x,y\in[w], w\in\mathcal W_k$, choose $z\in[w]$ such that $|f_k(z)|\cdot \nu_k([w])\leqslant 2\|f_k{\bf 1}_{[w]}\|_{k,1}$, then letting $ D_{[w]}(f) := \sup \{ |f(x)-f(y)|/d(x,y) : x,y \in [w]\}$
\begin{align*}
|f_k(x)+f_k(y)|\cdot \nu_k([w])&\leqslant (|f_k(x)-f_k(z)|+|f_k(y)-f_k(z)|+2|f_k(z)|)\cdot \nu_k([w])\\
&\leqslant 2r D_{[w]}(f_k)\nu_k([w])+4\|f_k{\bf 1}_{[w]}\|_{k,1}.
\end{align*}
Hence \begin{align*}
\| \mathcal D_r(f_k^2)\|_{k,1} &=\sum_{w\in \mathcal W_k}\nu_k([w]) D_{[w]}(f_k^2)
=\sum_{w\in \mathcal W_k}\nu_k([w]) \sup_{x,y\in[w]} \frac{|f_k^2(x)-f_k^2(y)|}{d(x,y)}\\
&=\sum_{w\in \mathcal W_k}\sup_{x,y\in[w]} \frac{|f_k(x)-f_k(y)|}{d(x,y)}|f_k(x)+f_k(y)|\cdot \nu_k([w])\\
&\leqslant  \sum_{w\in \mathcal W_k} D_{[w]}(f_k)\cdot \left(2r D_{[w]}(f_k)\nu_k([w])+4\|f_k{\bf 1}_{[w]}\|_{k,1}\right)\\
&\leqslant 2 \int (\mathcal D_r(f_k))^2 d\nu_k + 4 \int (\mathcal D_{r}(f_k)) |f_k| d \nu_k \\
&\leqslant 2 \|\mathcal D_r(f_k) \|_{k,2}^2 + 4 \|\mathcal D_r(f_k) \|_{k,2} \| f_k \|_{k,2}
\end{align*} is uniformly bounded. A similar bound can be obtained for $\|\mathcal D_r(h_k^2)\|_{k,1}$.
 The remaining terms in $\|\mathcal D_r(u_k^2)\|_{k,1}$ are also uniformly bounded, as
$$\|\mathcal D_r(h_{k+1}^2\circ T_k^1)\|_{k,1}\leqslant r^{-1}D(h_{k+1}^2)\leqslant 2r^{-1} D(h_{k+1})\|h_{k+1}\|_{\infty},$$
$$\|\mathcal D_r(f_kh_k)\|_{k,1}\leqslant \| \mathcal D_r(f_k)\|_{k,1}\|h_k\|_\infty+2D_r(h_k)\|f_k\|_{k,1}$$
and  similar estimates for $f_k\cdot (h_{k+1}\circ T_k^1)$ and $h_k\cdot (h_{k+1}\circ T_k^1)$ hold. Thus we have shown that $\|\mathcal D_r(u_k^2)\|_{k,1}$ is uniformly bounded. These uniform bounds imply that there exists a constant $C_1>0$ such that 
$$\mathbb E(F_n-1)^2=\sigma_n^{-4}\mathbb E\left(\sum_{k=0}^{n-1} \P_k^1\tilde u_k\circ T_0^{k+1}\right)^2\leqslant C_1 \sigma_n^{-4}\sum_{0\leqslant l\leqslant n-1}\mathbb E (U_l^2)= C_1\sigma_n^{-2}.$$
As $\sigma_n\to\infty$, $\mathbb E(F_n-1)^2\to 0$. We need to show that this convergence is an almost sure convergence. For $ \tilde C := \sup_k \mathbb E (U^2_k)$, let $m_n:=\inf \{t: \sigma_t^2\geqslant n^2\tilde C\}.$
Then $m_n<\infty$ and $$n^2\tilde C\leqslant \sigma_{m_n}^2\leqslant (n^2+1)\tilde C.$$
Since $\sum_n \mathbb E(F_{m_n}-1)^2\leqslant C_1\sum_n\sigma_{m_n}^{-2}<\infty$, $F_{m_n}\to 1$ a.s. by the Borel-Cantelli lemma.  Choose $\tilde n=\tilde n(n)$  such that $m_{\tilde n}\leqslant n\leqslant m_{\tilde n+1}$, then 
$$ F_{m_{\tilde n}}\frac{\tilde n^2}{(\tilde n+1)^2+1}\leqslant F_{m_{\tilde n}}\frac{\sigma^2_{m_{\tilde n}}}{\sigma^2_{m_{\tilde n+1}}}\leqslant F_n\leqslant  F_{m_{\tilde n+1}}\frac{\sigma^2_{m_{\tilde n+1}}}{\sigma_{m_{\tilde n}}^2}\leqslant F_{m_{\tilde n+1}}\frac{(\tilde n+1)^2+1}{\tilde n^2}.$$
Hence, $F_n\to 1$ a.s.

Moreover, $\sum_{n}\sigma_n^{-4}\mathbb E(|U_n|^{4})<\infty$ because $\sum_n s_n^{-4}$ and $\sup_n\|f_n^4\|_{n,1}$ are finite and because $\|U_n-f_n\circ T_0^n\|_\infty$ and $|\sigma_n-s_n|$ are both uniformly bounded.

Now we can use Theorem \ref{thm:rvsmtg} to find a sequence of independent centered Gaussian variables $\{Z_k\}$ with $\mathbb E(Z_k^2)=\mathbb E(U_k^2)$ such that 
$$\sup_{0\leqslant k\leqslant n-1}\left|\sum_{i=0}^k U_i-\sum_{i=0}^k Z_i\right|=o\left(\sqrt{\sigma^2_n \log\log \sigma^2_n}\right)\qquad \text{a.s.}$$
Since $|\sum_{i=0}^{k} f_i\circ T_0^i-\sum_{i=0}^{k} U_i|$ and $|\sigma_n-s_n|$ are both uniformly bounded, the statement of the theorem immediately follows.
  \end{proof}

\section{The law of the iterated logarithm}
We  apply the results of the proceeding section to continued fractions with restricted entries. In order to do so, recall that each irrational number $x \in [0,1]$ has a uniquely determined continued fraction $(x_0, x_1, x_2 \ldots)$ with $x_n \in \N$ such that 
\[ x = \llbracket x_0, x_1, \ldots \rrbracket := \frac{1}{ x_0 + \frac{1}{x_1 + \cdots}}. \]
For any word $w$, denote by $\llbracket w \rrbracket\subset [0,1]$ the subset corresponding to the cylinder $[w]$.
We now consider the following subsets of $[0,1]$ with restricted entries. That is, 
for a sequence of natural numbers $(\alpha_n : n \in \N_0)$ with $\lim_{n\to \infty} \alpha_n  = \infty$, 
we consider the set $X_\alpha := \{ x : x_n \geqslant \alpha_n \; \forall n \in \N_0 \}$ (where $\alpha=\llbracket \alpha_0,\alpha_1,\ldots\rrbracket$) and ask for a law of large numbers and the maximal growth rate of $x_n$ almost surely. In order to obtain these 
from the above results, we first analyze the geometric potential which, in particular, allows to determine the Hausdorff dimension and the measure class of $\nu_0$

\subsection{The geometric potential}
Let $(X_n)$ be the non-stationary shift space with $n$-th alphabet $\cW_n:= \{ l \in \N : l \geqslant \alpha_n\}$ and equipped with the potential functions 
\[\varphi_n:X_n \to (0,1), \quad (x_n,x_{n+1}, \ldots ) \mapsto  (\llbracket x_n,x_{n+1}, \ldots \rrbracket)^{2 \delta_n}, \]
for a sequence $(\delta_n)$ in $(0,1]$ to be determined below. Note that the $\varphi_n$ are geometric potentials as $\varphi_n = \varphi^{\delta_n}$ where $\varphi:=|1/S'|$ and $S$ referring to the continued fraction map 
$S(x) = 1/x \mod 1$. For $w = (w_n, \ldots, w_{n+k-1}) \in \cW_n^k$, set 
\[\Phi_w  := \prod_{l=0}^{k-1} \varphi_{n+l} \circ \tau_{(w_{n+l}, \ldots, w_{n+k-1})}.\]
Recall that $\tau_{(w_{n+l}, \ldots, w_{n+k-1})}: X_{n+k}\to X_{n+l}$ is the inverse of the map $T_{n+l}^{k-l}|_{[(w_{n+l}, \ldots, w_{n+k-1})]}$. As it is well known, the geometric potential $|1/S'|$ of the continued fraction map satisfies the Gibbs-Markov property. As $\delta_n \leqslant 1$, it follows from this that there exist $C\geqslant 1$ and $r \in (0,1)$ such that for all $k,n \in \N$,  $u \in \cW_0^n$, $w \in \cW_n^k$ and  and $x,y \in [w]$,  we have $|\log (\Phi_u(x) / \Phi_u(y))| \leqslant Cr^k$. 
In order to  simplify the notation, set
\[ \gamma_n := 2 \delta_n -1 \hbox{ and }
 \epsilon_n := (\alpha_{n} \alpha_{n-1})^{-1}.\]
\begin{lemma} \label{Lemma:Estimate_L(1)}
If $\gamma_n$ is given by $\gamma_n \alpha_n^{\gamma_n} =1$, then $\| \bL_0^n(\1)\|_\infty < \infty$. If 
 $\lim_n \alpha_n = \infty$, then  $|\bL_0^n(\1)(x)/\bL_0^n(\1)(y)-1| \ll \epsilon_n$ for all $x, y\in X_n$.  
\end{lemma}
\begin{remark}\label{remark:alpha vs gamma} Before giving the proof, we relate  $\alpha_n$ and $\gamma_n$. Recall that Lambert's $W$-function is defined as the inverse of $x \mapsto xe^x$ and observe that the inverse of $x \mapsto x^x$ for $x>0$ can be expressed in terms of the positive branch of $W$. That is, the inverse is equal to $x \mapsto e^{W(\log x)}$ and, as $\alpha_n = (1/\gamma_n)^{1/\gamma_n}$, 
$ \gamma_n = e^{-W(\log \alpha_n)}$. 
As $W(x) = \log x - \log\log x + O(\log\log x/\log x)$ (see (4.19) in \cite{CorlessGonnetHare:1996}) it follows from $\alpha_n \to \infty $ that
\begin{equation}
\label{eq:gamma_n}
\gamma_n \sim \frac{\log \log \alpha_n}{\log \alpha_n}.
\end{equation}
In particular, $\lim_n \gamma_n=0$ and $\lim_n \delta_n=1/2$. We also have $-\log\gamma_n\sim \log\log\alpha_n$.
\end{remark}
\begin{proof} 
The proof of the first part relies on the  observation that the image of $[k]$ under $(w) \mapsto \llbracket(w) \rrbracket$ is contained in $(1/(k+1), 1/k)$. Therefore, $\varphi_n \in (1/(k+1)^{2\delta_n}, 1/k^{2\delta_n})$ and, for $x \in X_n$,
\begin{align*}
\bL_0^n(\1)(x) & \leqslant 
\sum_{l_0 \geqslant \alpha_0, \ldots , l_{n-1} \geqslant \alpha_{n-1}} \frac{1}{l_0^{2\delta_0}} \cdots\frac{1}{l_{n-1}^{2\delta_{n-1}}} = 
 \prod_{k=0}^{n-1} \sum_{l= \alpha_k}^\infty \frac{1}{l^{2\delta_k}}\\
 & \leqslant   \prod_{k=0}^{n-1}  \frac{1}{2\delta_k -1} (\alpha_k -1)^{1-2\delta_k} 
    =   \prod_{k=0}^{n-1}  \frac{1}{\gamma_k} (\alpha_k -1)^{-\gamma_k}  =:b_+
\end{align*} 
The proof of the second assertion relies on a simple distortion estimate.  In order to do so, observe that $|(S^k)'(y)|\geqslant \prod_{j=1}^{k}\alpha_{l+j}^{2}$ for all $y \in X_{l+1}$. Hence, for $v \in \cW_{l+1}^k$ and ${z}, \tilde{z} \in X_{l+k+1}$, 
\[ |\tau_v(z) - \tau_v( \tilde{z}) | \leqslant  \diam(X_{l+k+1})  \prod_{j=1}^{k}\frac{1}{\alpha_{l+j}^{2}} = \frac{1}{\alpha_{l+k+1}}   \prod_{j=1}^{k}\frac{1}{\alpha_{l+j}^{2}}.\] 
Furthermore, for $w \in \cW_{l}$ and $x = \tau_w(y) \in X_l$, 
\[ \log \varphi_l(x) = \log |\tau_w'(y)|^{\delta_l} =  - 2 \delta_l \log(y + w),
 \quad   (\log \varphi_l \circ \tau_w)'(y) = {-2 \delta_l}/{(y + w)}.\]
As $w \geqslant \alpha_l$, it follows that $\log \varphi_l \circ \tau_w$ is Lipschitz continuous with Lipschitz constant $2\delta_l/\alpha_l $.
In particular, for $w \in \cW^n_0$ and $z,\tilde{z} \in X_{n}$,  
\begin{align*} (\ast) = \left| \log \Phi_w (z) - \log \Phi_w (\tilde{z}) \right| 
 & \leqslant   \frac{2\delta_n}{\alpha_{n}}  \sum_{l=0}^{n-1}  \frac{1}{\alpha_l}   \prod_{j=l+1}^{n-1} \frac{1}{ \alpha_{j}^{2}} \\
 & =  \frac{2\delta_n}{\alpha_{n}\alpha_{n-1}}  \left( 1 +  \sum_{l=0}^{n-2}  \frac{1}{\alpha_l \alpha_{n-1}}   \prod_{j=l+1}^{n-2} \frac{1}{ \alpha_{j}^{2}} \right)
\end{align*}
As $\lim_n \alpha_n = \infty$, it follows that the sum on the right hand side is convergent. Therefore,  $(\ast) \ll \epsilon_n$. 
Moreover, as $\lim_n \epsilon_n = 0$, $(\ast)$ is uniformly bounded. Hence, there exists $C\geqslant 1$ such that $|\Phi_w (z)/\Phi_w (\tilde{z})-1| \leqslant C \epsilon_n$, which then implies that $\Phi_w (z) \leqslant (1+C\epsilon_n)\Phi_w (\tilde{z})$ and $\bL_0^n(\1)(z) \leqslant (1+C\epsilon_n) \bL_0^n(\1)(\tilde{z})$. The second statement follows from  this.
  \end{proof}
Hence Theorems \ref{theo:exponential decay} and \ref{theo:exponential decay for L1} are applicable, providing existence of $(\nu_n)$ and exponential decay. Furthermore, as $\llbracket \, \cdot \, \rrbracket: X_n  \to [0,1]$ is injective, we identify each $\nu_n$ with a probability measure supported on $\llbracket X_n \rrbracket \subset [0,1] \setminus \mathbb{Q}$. $X_0$ is identified with $X_\alpha$ as well.

\subsection{Expected value and variance of the logarithm.}

Let $\tlog  :(0,1] \to \R$ be given by $\tlog(x):= \log n$ for $x \in (1/(n+1),1/n]$, $n \in \N$. The following lemma describes the asymptotic behavior of the expected values and  variances of  $\tlog$ with respect to the measure $\nu_n$.

\begin{lemma}\label{lemma:Estimate_log} Under the assumptions of Lemma \ref{Lemma:Estimate_L(1)},
\begin{equation*}
\lim_{n\to \infty}  \int \left(\gamma_n \tlog+\log\gamma_n\right) d\nu_n = 1 ~~\text{ and }~~   \lim_{n\to \infty} \int \left(\gamma_n \tlog + \log\gamma_n-1\right)^2  d\nu_n = 1.
\end{equation*}
 If, in addition, $\sum_n 1/ \alpha_n<\infty$ then 
\begin{align}
\label{eq:expectation} 
\sum_{n=0}^\infty   \left|\int (\gamma_n \tlog+\log\gamma_n) d\nu_n -1\right| < \infty\\
\label{eq:variance} 
\sum_{n=0}^\infty \left|\int (\gamma_n \tlog+\log\gamma_n-1)^2 d\nu_n -1\right|< \infty
\end{align} 
\end{lemma}

\begin{proof}
It follows from Lemma \ref{Lemma:Estimate_L(1)} that there exists a constant $C$ such that for $n\in \mathbb N_0$ with $C \epsilon_n<1/3$ and $x\in X_{n+1}$,
\begin{align*}
\P_n^1(\tlog) (x) & = \frac{\bL_n^1(\tlog\cdot \bL_0^n(\1))}{\bL_0^{n+1}(\1)} (x) 
\leqslant \frac{(1+C\epsilon_n)\bL_n^1(\tlog)}{(1-C\epsilon_n)\bL_n^1(\1)} (x) \\
& \leqslant  \frac{(1+C\epsilon_n)\sum_{k \geqslant \alpha_n} {(\log k)/k^{2\delta_n}}}{(1-C\epsilon_n)\sum_{k\geqslant \alpha_n} {1/(k+1)^{2\delta_n}}}\\
& \leqslant \frac{1+C\epsilon_n}{1-C\epsilon_n} \frac{\gamma_n \log (\alpha_n -1) + 1}{\gamma_n^2 (\alpha_n -1)^{\gamma_n}} / \frac{1}{\gamma_n (\alpha_n+1)^{\gamma_n}}\\
& \leqslant \frac{1+C\epsilon_n}{1-C\epsilon_n}\left(\frac{\alpha_n+1}{\alpha_n-1}\right)^{\gamma_n} \frac{1 -  \log \gamma_n}{\gamma_n}\\
& \leqslant (1+3C\epsilon_n)\left(1+\frac{2\gamma_n}{\alpha_n-1}\right)\frac{1-\log\gamma_n}{\gamma_n}.
\end{align*}
The estimate from below is similar:
\begin{align*}
\P_n^1(\tlog)(x) & \geqslant \frac{1-C\epsilon_n}{1+C\epsilon_n}\frac{\alpha_n}{\alpha_n+1}\left(\frac{\alpha_n-1}{\alpha_n+1}\right)^{\gamma_n}\frac{1-\log \gamma_n}{\gamma_n} \\
&\geqslant \left(1-2C\epsilon_n\right)\left(1-\frac{4 \gamma_n }{\alpha_n}\right)\frac{1-\log\gamma_n}{\gamma_n}.
\end{align*}
As $-\log\gamma_n\sim \log\log\alpha_n$, it follows that $-\epsilon_n\log\gamma_n \sim \log\log\alpha_n/(\alpha_n \alpha_{n-1}) \to 0$ as $\alpha_n \to \infty$. Hence, $\|\gamma_n\P_n^1 (\tlog)+\log\gamma_n \|_\infty < \infty$.  It now follows from Theorem \ref{theo:exponential decay} (3) that for any sequence $x_k\in X_{n+k+1},$
\[  \int (\gamma_n \P_n^1(\tlog)+\log\gamma_n) d\nu_{n+1} = \lim_{k\to \infty} \P_{n+1}^k\left(\gamma_n \P_n^1(\tlog)+\log\gamma_n\right)(x_k). \]
Now we can employ $\P_n^k(\1)=\1$ and the estimates for $\P_n^1(\tlog)$ to deduce
$$\lim_{n\to\infty}\int (\gamma_n \tlog+\log\gamma_n) d\nu_n = \lim_{n\to\infty} \int \left (\gamma_n \P_n^1(\tlog)+\log\gamma_n\right) d\nu_{n+1}=1.$$
By the same argument, for all $x\in X_{n+1}$,
\begin{align*}
\P_n^1(\tlog^2)(x) 
&  \leqslant \frac {(1+C\epsilon_n)\sum_{k \geqslant \alpha_n} \log^2(k)/k^{\gamma_n +1}}{(1-C\epsilon_n)\sum_{k\geqslant\alpha_n} 1/(k+1)^{\gamma_n+1}} \\
& \leqslant \frac{1+C\epsilon_n}{1-C\epsilon_n}  \frac{\left(1 + \gamma_n \log(\alpha_n -1)\right)^2 +1}{\gamma_n^3 (\alpha_n -1)^{\gamma_n}} /\frac{1}{\gamma_n(\alpha_n+1)^{\gamma_n}}  \\
& \leqslant \frac{1+C\epsilon_n}{1-C\epsilon_n}\left(\frac{\alpha_n+1}{\alpha_n-1}\right)^{\gamma_n}\frac{(1 -  \log \gamma_n)^2 +1}{\gamma_n^2}
\end{align*}   
and $\P_n^1(\tlog^2)(x)\geqslant \frac{1-C\epsilon_n}{1+C\epsilon_n}\frac{\alpha_n}{\alpha_n+1}\left(\frac{\alpha_n-1}{\alpha_n+1}\right)^{\gamma_n}\frac{(1-\log \gamma_n)^2+1}{\gamma_n^2}.$  The limit follows analogously.

In order to prove \eqref{eq:expectation}, the estimates above indicate that it suffices to  verify the finiteness of $\sum -\epsilon_n \log\gamma_n$ and $\sum - (\gamma_n \log\gamma_n)/\alpha_n$. However, by remark \ref{remark:alpha vs gamma}, $\sum_n 1/ \alpha_n < \infty$ and $K$ sufficiently large,
\begin{align*} 
- \sum_{n=K}^\infty \epsilon_n \log\gamma_n & \asymp \sum_{n=K}^\infty \frac{\log\log\alpha_n}{\alpha_n \alpha_{n-1}} \leqslant  \sum_{n=K}^\infty \frac{1}{\alpha_{n-1}} < \infty, \\
-  \sum_{n=K}^\infty \frac{\gamma_n \log\gamma_n}{\alpha_n} & \asymp
\sum_{n=K}^\infty  \frac{(\log\log\alpha_n)^2}{\alpha_n \log\alpha_n}  \leqslant  \sum_{n=K}^\infty \frac{1}{\alpha_{n}} < \infty.
\end{align*}
The same arguments prove \eqref{eq:variance}.
  \end{proof}

\begin{remark}
In the course of proving \eqref{eq:expectation} we have actually shown that 
\begin{equation}\label{eq:finitesumP}
\sum_{n=0}^\infty\|\gamma_n\P_n^1\tlog+\log\gamma_n-1\|_\infty<\infty.
\end{equation}
\end{remark}

\begin{lemma} \label{lemma:var4} Under the assumptions of Lemma \ref{Lemma:Estimate_L(1)},
$$\lim_{n\to\infty}\int (\gamma_n\tlog+\log \gamma_n-1)^4 d\nu_n=9.$$
\end{lemma}

\begin{proof}
The proof uses the same arguments in the previous lemma. There are only few additional calculations as presented below. For any $\gamma>0$,
$$\int \frac{\log^4(x)}{x^{1+\gamma}}dx=\frac{-1}{\gamma^5 x^\gamma}\left((\gamma\log(x)+1)^4+6(\gamma\log(x)+1)^2+8(\gamma\log(x)+1)+9\right).$$
Let $F_n(x)$ be the polynomial function such that $$- \gamma^{n+1}x^\gamma \int \frac{\log^n(x)}{x^{1+\gamma}} d(x) = F_n(\gamma\log(x)+1),$$
and $a_n:=F_n(0)$. For any function $f$, let $f_n(x):=f^n(x)-F_n(1-\log\gamma)$. Then
\begin{align*}&\quad \left(f(x)-(1-\log\gamma)\right)^4\\
&=f_4(x)-4f_1(x)(1-\log\gamma)^3+6f_2(x)(1-\log\gamma)^2-4f_3(x)(1-\log\gamma)+a_4.\end{align*}
  \end{proof}

\subsection{Law of the iterated logarithm} Applying Theorem \ref{theo:invpp} to $X_\alpha$, we obtain the law of the iterated logarithm for the continued fraction expansions with restricted entries.

\begin{theorem} \label{prop:Khinchine-Stackelberg}
Assume that  $\sum_n 1/\alpha_n<\infty$. For $\nu_0$-a.e. $x\in X_\alpha$ with  continued fraction expansion $(x_0, x_1, \ldots)$, and $\gamma_n$ given by $\gamma_n\alpha_n^{\gamma_n}=1$,
\begin{equation} \label{eq:LIL-Digits}
\limsup_{n\to\infty} \frac{\sum_{k=0}^{n-1}(\gamma_k\log(x_k/\alpha_k)-1)}{\sqrt{2n\log\log n}}=1.
\end{equation}
The $n$-th approximant $p_n(x)/q_n(x) = \llbracket  x_0, \ldots, x_{n-1}, \infty \rrbracket $ satisfies a.s., with $\gamma_{-1}:=0$
\begin{equation} \label{eq:LIL-Approximant}\limsup_{n\to \infty} \frac{\sum_{k=0}^{n-1} \left((\gamma_{k-1}-\gamma_{k}) \log \frac{q_{n-k}(S^{k}(x))}{q_{n-k}(S^{k}(\alpha))} - 1\right)}{\sqrt{2 n\log\log n }} = 1.\end{equation}
\end{theorem}
\begin{proof} In order to prove \eqref{eq:LIL-Digits}, 
for $n\in\mathbb N_0$ and $x\in X_n$, define $f_n: X_n\to\R$ by $$f_n(x):=\gamma_n\tlog(x)+\log\gamma_n=\gamma_n\log(x_n/\alpha_n)>0.$$ Clearly $D_r(f_n)=0$. Lemma \ref{lemma:Estimate_log} shows that $f_n\in\mathcal H_n^{loc}$ and, with $\tilde f_n:=f_n-\int f_n d\nu_n$,  that
$$\sum_{n=0}^\infty\left|\int f_n d\nu_n-1\right|<\infty, \quad \sum_{n=0}^\infty\left|\int\tilde f_n^2d\nu_n-1\right|<\infty.$$ 
We will check the conditions in Theorem \ref{theo:invpp} for $\{\tilde f_n\}$.
Lemma \ref{lemma:var4} implies that $$\sup_n \int \tilde f_n^4d\nu_n<\infty.$$
For $s_n^2:=\int (\sum_{k=0}^{n-1} \tilde f_k\circ T_0^k)^2 d\nu_0$, we have
\begin{align*}
s_n^2-n=\sum_{k=0}^{n-1}\left(\int \tilde f_k^2 d\nu_k-1\right)+2\sum_{0\leqslant k<j\leqslant n-1}\int \tilde f_k\cdot \tilde f_j\circ T_k^{j-k} d\nu_k.
\end{align*}
By Theorem \ref{theo:exponential decay for L1} for some constant $C$ and $s\in (0,1)$, when $j-k>\sqrt{n}$ 
\begin{align*}
\sum_{\substack{0\leqslant k<j\leqslant n-1\\ j-k>\sqrt n}}\left|\int \tilde f_k\cdot \tilde f_j\circ T_k^{j-k} d\nu_k\right|&=\sum_{\substack{0\leqslant k<j\leqslant n-1\\ j-k>\sqrt n}}\left|\int \P_k^{j-k}\tilde f_k\cdot \tilde f_j d\nu_j\right|\\
&\leqslant C\sum_{\substack{0\leqslant k<j\leqslant n-1\\ j-k>\sqrt n}}s^{\sqrt {j-k}} \|f_k\|_{k,1}\cdot \|f_j\|_{j,1}\ll s^{\sqrt n}n,
\end{align*}
while by \eqref{eq:finitesumP} when $1\leqslant j-k\leqslant \sqrt n$ 
\begin{align*}\sum_{\substack{0\leqslant k<j\leqslant n-1\\1\leqslant j-k\leqslant \sqrt{n}}}\left|\int \tilde f_k\cdot \tilde f_j\circ T_k^{j-k} d\nu_k\right|&=\sum_{\substack{0\leqslant k<j\leqslant n-1\\1\leqslant j-k\leqslant\sqrt{n}}}\left|\int \P_k^{1}\tilde f_k\cdot \tilde f_ j\circ T_{k+1}^{j-k-1}\nu_{k+1}\right|\\
&\leqslant \sum_{k=0}^\infty \|\P_k^1\tilde f_k\|_\infty\cdot\sum_{j=1}^{\sqrt n}\|\tilde f_j\|_{j,1}\ll \sqrt n.
\end{align*}
It follows from these estimates that
$${s_n^2}/{n}=1+O(1/n).$$
Now we can obtain by Theorem \ref{theo:invpp} that there exists a sequence $(Z_n)$ of independent centered Gaussian random variables such that 
\begin{align}\label{eq:invp}
 \sup_{0\leqslant k \leqslant n-1} \left| \sum_{i=0}^k \tilde f_i\circ T_0^i - \sum_{i=0}^k Z_i \right| = o(\sqrt{n \log \log n}) \hbox{ a.s.,} \\
\nonumber  \sup_n\left|\sqrt{\sum_{k=0}^{n-1} \Var(Z_k)}-\sqrt n\right|<\infty, \quad 
 \sup_n| \|\tilde f_n\|_{n,2}-\sqrt {\Var (Z_n)}|<\infty.
\end{align} 
 Hence, $\sup_n\Var (Z_n)<\infty$ as well. This implies that  
$ \sum_{k=0}^{n-1}\Var(Z_k)\to\infty$  and  $\frac{\Var (Z_{n-1})}{\sum_{k=0}^{n-1}\Var(Z_k)}\ll \frac1n $. Then one can employ the LIL for weighted averages by Chow and Teicher \cite[Theorem 1]{ChowTeicher:1973} along with \eqref{eq:invp} to finish the proof of \eqref{eq:LIL-Digits}.

For the proof of \eqref{eq:LIL-Approximant}, observe that for $x=\llbracket x_0,x_1,\ldots\rrbracket$, $S^k(x)=\llbracket x_k,x_{k+1},\ldots\rrbracket$. Hence, for $x\in X_\alpha$, $-\gamma_k\log(S^k(x))=\gamma_k\log(x_k)+o(1/\alpha_k)$. In particular, it follows from \eqref{eq:LIL-Digits} that 
\begin{equation}\label{eq:general-LIL}
\limsup_{n\to\infty} \frac{\sum_{k=0}^{n-1}(\gamma_k\log\frac{S^k(\alpha)}{S^k(x)}-1)}{\sqrt{2n\log\log n}}=1. \hbox{ a.s..} \end{equation} 
Furthermore, it is well known that (e.g. \cite[page 85]{EinsiedlerWard:2011}) for each $x \in [0,1]\setminus \mathbb{Q}$, 
\[\left| \sum_{k=0}^{n-1}   \log S^{k}(x) + \log q_n(x) \right| \leqslant 4 + 2\log 2. \]
As $p_k(x)=q_{k-1}(S(x))$, it follows that 
\[\sum_{k=0}^{n-1} \gamma_{k}\log\frac{p_{n-k}(S^{k}(x))}{q_{n-k}(S^{k}(x))}=\sum_{k=0}^{n-1} (\gamma_{k-1}-\gamma_{k})\log (q_{n-k}(S^{k}(x))).\]
The remaining assertion follows then from \eqref{eq:general-LIL}. 
  \end{proof}

\section{Hausdorff dimension and measure} We are now in position to prove
 the following proposition relating  $\nu_0$ with the $\delta$-dimensional Hausdorff measure $\mu_\delta$ on $X_\alpha$.

\begin{proposition}\label{prop:HD-dimension} Assume that $\sum_n  1/\alpha_n <\infty$ and that there exists $\delta >0$ with 
\[ \limsup_{n \to \infty} \frac{1}{n}\left(  \frac{\log \alpha_n}{2} \left( \frac{(1+ \delta)\log n}{\log \log \alpha_n} +1  \right)   - \sum_{k=0}^{n-1} \log \log \alpha_k  \right) <1, \]
then $X_\alpha$ has Hausdorff dimension $1/2$ and $\nu_0$ is absolutely continuous with respect to 
$\mu_{1/2}$, but there exist subsets  $A$ of $X_\alpha$ with $\nu_0(A) =0 $ and  $\mu_{1/2}(A)> 0$. Furthermore, 
$\mu_{1/2} ( \llbracket w \rrbracket)= \infty$ for each finite word $w$.
\end{proposition}

\begin{proof} 
We begin with the upper bound for the Hausdorff dimension which immediately follows from Good's result in \cite{Good:1941} mentioned above. However, in order to reveal the geometric nature of $\nu_0$, we prove the upper bound in the following way. For $\epsilon >0$ choose $n \in \N$ such that $\delta_k = (\gamma_k + 1)/2  \leqslant 1/2 + \epsilon$ for all $k \geqslant n$ (see Remark \ref{remark:alpha vs gamma}). As $\varphi$ is a geometric potential of bounded distortion, we have, for $w \in \cW_0^n$ and $u =(u_1, \ldots, u_m) \in \cW_n^m $ for some $m$ and $v:= (u_1, \ldots, u_{m-1}, u_m+1)$ that 
\begin{align*}
\frac{\nu_0( [wu] \cup [wv])}{(\diam  (\llbracket wu \rrbracket \cup \llbracket wv \rrbracket))^{1/2 + \epsilon}} 
& \gg   \frac{\Phi_{w}^{\delta_0}  \left(\Phi_u^{1/2 + \epsilon} + \Phi_v^{1/2 + \epsilon}\right) }{\Phi_{wu}^{1/2 + \epsilon} \left(\Phi_u^{1/2 + \epsilon} + \Phi_v^{1/2 + \epsilon}\right)}
=\frac{\Phi_{w}^{\delta_0}}{\Phi_{w}^{1/2 + \epsilon}}.
\end{align*}
Hence, by considering disjoint covers of $X_\alpha \cap \llbracket w \rrbracket$ by sets of type $[wu] \cup [wv]$, one obtains a uniform upper bound for $\mu_{1/2 + \epsilon}(\llbracket w\rrbracket )$. Hence the Hausdorff dimension of $\llbracket w \rrbracket$ is smaller than or equal to $1/2+\epsilon$. This then implies that the Hausdorff dimension of $X_\alpha$ is bounded by $1/2+\epsilon$ and as  $\epsilon$ is arbitrary, smaller than or equal to $1/2$. 

The lower bound requires estimates for the diameters of cylinder sets and their unions. As $\diam \llbracket X_n\rrbracket \sim 1/\alpha_{n}$,  
$\varphi$ is a geometric potential and $\log \Phi_w$ is H\"older continuous with respect to uniform constants for each $w=(w_0, \ldots, w_{n-1}) \in \cW_0^n$, we have 
\[
\diam ( \llbracket w \rrbracket \cap X_\alpha)\sim \Phi_w(0) /\alpha_{n} .
\]
For $x \in  \llbracket w \rrbracket$ and $0 \leqslant k \leqslant n-1$, let $\xi_k(x) := \left(\alpha_k \sqrt{\varphi(S^k(x))}\right)^{\gamma_k}$.  
Using $\gamma_k \alpha_k^{\gamma_k}= 1 $, one obtains the following upper bound for the ratio of a cylinder and its diameter. 
\begin{align*}
\frac{\nu_0( \llbracket w \rrbracket )}{\sqrt{\diam ( \llbracket w \rrbracket  \cap X_\alpha))}}
& 
\asymp \frac{\prod_{k=0}^{n-1} \varphi_k(S^k(x))}{\prod_{k=0}^{n-1}  \varphi(S^k(x))^{1/2}} \cdot \alpha_{n}^{1/2}
\\
& \asymp
\prod_{k=0}^{n-1}  \frac{  (\alpha_k^2 \varphi(S^k(x)))^{\delta_k}}{( \alpha_k^2 \varphi(S^k(x))^{1/2}}  \cdot
 \alpha_k^{1-2\delta_k}\alpha_{n}^{1/2} = \alpha_{n}^{1/2} \prod_{k=0}^{n-1} \gamma_k \cdot \xi_k(x).
\end{align*}
Furthermore, for $t > s \geqslant \alpha_{n}$ and a union of adjacent cylinders $A= \bigcup_{u=s}^t  \llbracket wu \rrbracket$, it follows from 
\begin{align*}
\diam (A)  \geqslant \diam (A \cap X_\alpha) > \diam ( \textstyle \bigcup_{u=s+1}^{t}  \llbracket wu \rrbracket) 
\end{align*}
that $\diam (A)  \asymp  \diam (A \cap X_\alpha)$. 
Then if  $x \in A$, we obtain that
\begin{align*}
\frac{\nu_0( A)}{\sqrt{\diam (A \cap X_\alpha))}} & \asymp
\frac{\prod_{k=0}^{n-1} \varphi_k(S^k(x))}{\prod_{k=0}^{n}  \varphi(S^k(x))^{1/2}} 
\cdot
\frac{\sum_{l=s}^t l^{-2\delta_{n}}}{\left(\sum_{l=s}^t l^{-2}\right)^{1/2}} \\
& \asymp \prod_{k=0}^{n-1}  (\gamma_k \xi_k(x))
 \cdot  
\frac{  \frac{1}{\gamma_{n}}   \left( {s^{-\gamma_{n}}} - {t^{-\gamma_{n}}} \right) }{ \left( s^{-1} - t^{-1}  \right)^{1/2} } \\
& \leqslant \prod_{k=0}^{n-1} 
(\gamma_k \xi_k(x))  \cdot  \alpha_{n}^{\gamma_n} \cdot  s^{1/2 - \gamma_{n}} \\
& \leqslant  x_{n}^{1/2} \prod_{k=0}^{n-1} \gamma_k \cdot \xi_k(x).
\end{align*}
It follows from the construction of $\nu_0$  that 
\begin{align*}
\nu_{0}(\{x: x_{n}^{1/2} \geqslant b\})=\nu_{n}(\{x=(x_n,\ldots): x_n^{1/2}\geqslant b\}) \asymp ( \alpha_{n}/b^2)^{\gamma_{n}}.
\end{align*}  
Hence, let $b_n := (\alpha_n  n^{(1+\delta)/\gamma_n})^{1/2}$ for some $\delta > 0$, then   
 $\sum_n \nu_{0}(\{x: x_n^{1/2} \geqslant b_n\}) < \infty$  and for $\nu_0$-a.e. $x$, by the Borel-Cantelli Lemma, $x_{n}^{1/2} \geqslant b_{n}$ happens only finitely many times. In particular, for fixed $x$ and $n$ sufficiently large,
\[ {\nu_0( A)}/{\sqrt{\diam (A \cap X_\alpha)}} \ll   n^{(1+\delta)/2\gamma_{n}} \cdot  \alpha_{n}^{1/2} \prod_{k=0}^{n-1} \gamma_k \cdot \xi_k(x) =: \Theta_n(x). \] 
It follows from \eqref{eq:general-LIL} that $\lim_n (\log \prod_{k=0}^{n-1} \xi_k(x))/n = -1$ a.s. Hence, $\lim  \Theta_n = 0 $ a.s. provided that 
\[ 0 > \limsup_{n \to \infty} \frac{1}{n} \log \Theta_n =  - 1 +  \limsup_{n \to \infty} \frac{1}{n}\left( \sum_{k=0}^{n-1} \log \gamma_k + \frac{1}{2\gamma_n} \log \frac{n^{1+ \delta}}{ \gamma_n}   \right). \]
However, it easily follows from the relation between $(\alpha_n)$ and $(\gamma_n)$ given in Remark \ref{remark:alpha vs gamma} that the above is equivalent to 
\[ \limsup_{n \to \infty} \frac{1}{n}\left(  \frac{\log \alpha_n}{2} \left( \frac{(1+ \delta)\log n}{\log \log \alpha_n} +1  \right)   - \sum_{k=0}^{n-1} \log \log \alpha_k  \right) <1. \]
Now assume that the condition holds. 
We have therefore that for $\nu_0$-a.e. $x$, any sequence of cylinders $A_n$ of the form $\cup_{u=s}^{t}\llbracket wu\rrbracket$ with $w=(x_0,\ldots, x_{n-1})$ and $\alpha_n\leqslant s\leqslant x_n\leqslant t\leqslant\infty$, $$\lim_{n\to\infty}\frac{\nu_0(A_n)}{\sqrt{\diam(A_n \cap X_\alpha)}}=0.$$ Fix a ball $B_r(x)$ of radius $r$ around $x$. Choose the minimal $n$ such that $\llbracket \tilde w\rrbracket\subset B_r(x)$ for $\tilde w=(x_0,\ldots,x_{n-1})$ and choose the maximal union of cylinders $A_n$ contained in $B_r(x)$ of the form $\cup_{u=s}^t\llbracket w u\rrbracket$ with $w=(x_0,\ldots, x_{n-2})$ and $\alpha_{n-1}\leqslant s\leqslant x_{n-1}\leqslant t\leqslant\infty$. 
Let $\check w=(x_0,\ldots,x_{n-3})$ and $v=x_{n-2}\pm1$. Note that for some distortion constant $C$, $\textrm{dist}(\llbracket \check w x_{n-2}\rrbracket \cap X_\alpha, \llbracket \check w v\rrbracket \cap X_\alpha)\geqslant \frac 1C (1-\frac{1}{\alpha_{n-1}})\diam \llbracket  w\rrbracket $. So $B_r(x)\cap X_\alpha=A_n$ provided that $r/\diam\llbracket w\rrbracket<\frac1{4C}$. In this case, $$\frac{\nu_0(B_r(x))}{r^{1/2}}=\frac{\nu_0(A_n)}{r^{1/2}}\leqslant \sqrt 2\frac{\nu_0(A_n)}{\sqrt{\diam(A_n)}}.$$
Suppose $r/\diam\llbracket  w\rrbracket\geqslant 1/(4C)$. The choice of $n$ implies $\diam\llbracket  w\rrbracket>r$, hence for large $n$, $\diam\llbracket \check w v\rrbracket>r$  and $B_r(x)\subset \llbracket w\rrbracket\cup \llbracket\check w v\rrbracket$. Meanwhile, for $y\in[\check wv]$, \[\nu_0(\llbracket w\rrbracket)/\nu_0(\llbracket \check wv\rrbracket)\asymp \prod_{k=0}^{n-2}\varphi_k(S^k(x))/\varphi_k(S^k(y))<\infty.\] We deduce that
$$\frac{\nu_0(B_r(x))}{r^{1/2}}\ll \frac{\nu_0(B_r(x))}{\sqrt{\diam\llbracket  w\rrbracket}}\ll \frac{\nu_0(\llbracket w\rrbracket)}{\sqrt{\diam(\llbracket w\rrbracket\cap X_\alpha)}}.$$
Then the above implies that
\[ \lim_{r \to 0} \frac{\nu_0(B_r(x))}{r^{1/2}} =0 \]
for $\nu_0$-a.e. $x$. Now assume that by Egorov's theorem $\Omega \subset X_\alpha$ is chosen such that the above limit is uniform. Therefore (for instance by \cite[Theorem 5.7]{Mattila:1995}) the Hausdorff dimension of $\Omega$ is at least $1/2$, and so is $X_\alpha$. It also implies that $\nu_0 \ll \mu_{1/2}$.
  \end{proof}

\section{Proof of Theorem \ref{theo:easy_bounds}} \label{sec:proof}
For the proof of Theorem \ref{theo:easy_bounds}, it suffices to show that the bounds on the growth of $(\alpha_n)$ imply the conditions 
of Proposition \ref{prop:HD-dimension}. That is, it remains to check that for some $\delta>0$
\begin{equation} \label{eq:condition-HDdimension}
\limsup_{n \to \infty} \frac{1}{n}\left(  \frac{\log \alpha_n}{2} \left( \frac{(1+ \delta)\log n}{\log \log \alpha_n} +1  \right)   - \sum_{k=0}^{n-1} \log \log \alpha_k  \right) <1. 
\end{equation}
If $n/C \leqslant \alpha_n \leqslant C \lambda^{n}$ for some $\lambda >1$ and $C \geqslant 1$, then  
\begin{align*}
\frac{\log \alpha_n}{2n} \left( \frac{(1+ \delta)\log n}{\log \log \alpha_n} +1  \right)  
&= 
\frac{\log \alpha_n}{2n\log \log \alpha_n}(1+ \delta)\log n +   \frac{\log \alpha_n}{2n} \\
& \leqslant \frac{\log C \lambda^{n}}{2n\log \log C \lambda^{n}}(1+ \delta)\log n +   \frac{\log C \lambda^{n}}{2n}\\
&\xrightarrow{n \to \infty} (1+\delta/2)\log \lambda. 
\end{align*}
As  $\log \log (n/C)  \sim \log \log n$ and $\int \log \log x dx = x \log\log x - \int_e^x (\log t)^{-1} dt$,  the lower bound implies that
\begin{align*}
\frac{1}{n} \sum_{k=0}^{n-1} \log \log \alpha_k 
\gg  \log\log n - \frac{1}{n} \int_e^n dt
 \xrightarrow{n \to \infty} \infty.
\end{align*} 
Hence, \eqref{eq:condition-HDdimension} is satisfied whenever $n \ll  \alpha_n \ll \lambda^{n}$.
It remains to analyse the case $\lambda^{n}\ll  \alpha_n \ll n^{(1-\epsilon)n}$ for some  $\lambda >1$ and  $\epsilon >0$. These bounds imply that 
\begin{align*}
\frac{1}{n} \sum_{k=0}^{n-1} \log \log \alpha_k &
\gg  \frac{1}{n} \sum_{k=0}^{n-1} \log n \sim \log n, \\
\frac{\log \alpha_n}{2n} \left( \frac{(1+ \delta)\log n}{\log \log \alpha_n} +1  \right) 
&\ll  
\frac{(1-\epsilon)(1+ \delta)(\log n)^2}{ 2(\log n + \log \log n + \log(1-\epsilon))} + 
\frac{(1-\epsilon)\log n}{2}\\
& \sim \frac{(1-\epsilon)(2 + \delta)}{2} \log n
\end{align*}
In particular, \eqref{eq:condition-HDdimension} is satisfied for $\delta :=\epsilon$. \qed

\section*{Acknowledgements} The first author acknowledges support by PQ 312632/2018-5 and Universal 426814/2016-9 of CNPq. The second was supported by the PNPD program of CAPES and then by Fapesp grant \#2018/15088-4.

\end{document}